\documentclass[11pt]{article}
\usepackage[doc]{optional}
\usepackage{enumitem}
\usepackage{mathtools,esvect}
\usepackage{cite}
\usepackage{color}
\usepackage{float}
\usepackage{subcaption}
\usepackage{soul}
\definecolor{labelkey}{rgb}{0,0.08,0.45}
\definecolor{refkey}{rgb}{0,0.6,0.0}

\usepackage[left]{lineno}
\usepackage{blindtext}

\usepackage{mathpazo}
\usepackage{amsmath}
\usepackage{amssymb}
\usepackage{theorem}
\usepackage{fancyhdr}

\usepackage[margin=1.0in]{geometry}

\newcommand{\Tau}{\mathrm{T}}

\newcommand{\menge}[2]{\big\{{#1}~\big |~{#2}\big\}}

\newcommand{\fenv}[1]%
{\ensuremath{\,\overrightarrow{\operatorname{env}}_{#1}}}
\newcommand{\benv}[1]%
{\ensuremath{\,\overleftarrow{\operatorname{env}}_{#1}}}
\newcommand{\emp}{\ensuremath{\varnothing}}

\newcommand{\scal}[2]{\left\langle{#1},{#2}  \right\rangle}

\newcommand{\HH}{\ensuremath{\mathcal H}}

\newcommand{\RR}{\ensuremath{\mathbb R}}

\newcommand{\zer}{\ensuremath{\operatorname{zer}}}

\newcommand{\NN}{\ensuremath{\mathbb N}}

\newcommand{\dom}{\ensuremath{\operatorname{dom}}}

\newcommand{\gr}{\ensuremath{\operatorname{Gph}}}

{\begin{list}{}{%
\settowidth{\labelwidth}{\textrm{#1~}}%
\setlength{\leftmargin}{\labelwidth+\labelsep}}}
{\end{list}}
\newtheorem{theorem}{Theorem}[section]
\newtheorem{lemma}[theorem]{Lemma}
\newtheorem{corollary}[theorem]{Corollary}
\newtheorem{proposition}[theorem]{Proposition}
\newtheorem{definition}[theorem]{Definition}
\theoremstyle{plain}{\theorembodyfont{\rmfamily}
}
\theoremstyle{plain}{\theorembodyfont{\rmfamily}
}
\theoremstyle{plain}{\theorembodyfont{\rmfamily}
}
\theoremstyle{plain}{\theorembodyfont{\rmfamily}
}

\theoremstyle{plain}{\theorembodyfont{\rmfamily}
\newtheorem{remark}[theorem]{Remark}}

\def\doi{DOI}

\newcounter{count}

\allowdisplaybreaks

\newcommand{\la}{{\langle}}
\newcommand{\ra}{{\rangle}}

\begin{document}

\title{\textrm{On the Weak and Strong Convergence of a Conceptual Algorithm for Solving Three Operator Monotone Inclusions}}
\author{
Yunier\ Bello-Cruz\thanks{Department of Mathematical Sciences, Northern Illinois University. Watson Hall 366, DeKalb, IL, USA - 60115. E-mail:
\texttt{yunierbello@niu.edu.}}\and Oday Hazaimah\thanks{Department of Mathematical Sciences, Northern Illinois University. Watson Hall 366, DeKalb, IL, USA - 60115. E-mail:
\texttt{oday@niu.edu.} }}

\date{\today}

\maketitle \thispagestyle{fancy}

\begin{abstract} \noindent In this paper, a conceptual algorithm modifying the forward-backward-half-forward (FBHF) splitting method for solving three operator monotone inclusion problems is investigated. The FBHF splitting method adjusts and improves Tseng's forward-backward-forward (FBF) splitting method when the inclusion problem has a third-part operator that is cocoercive. The FBHF method recovers the FBF iteration (when this aforementioned part is zero), and it also works without using the widely used Lipschitz continuity assumption.   
The conceptual algorithm proposed in this paper also has those advantages, and it derives two variants (called Method 1 and Method 2) by choosing different projection (forward) steps. Both proposed methods also work efficiently without assuming the Lipschitz continuity and without directly using the cocoercive constant. Moreover, they have the following desired features:  (i) very general iterations are derived for both methods, recovering the FBF and the FBHF iterations and allowing possibly larger stepsizes if the projection steps are over-relaxing; and (ii) strong convergence to the best approximation solution of the problem is proved for Method 2. To the best of our knowledge, this is the first time that an FBF-type method converges strongly for finding the best approximation solution of the three operator monotone inclusion.     
\end{abstract}



\section{Introduction} \label{sec:int}
In this work, we present a conceptual algorithm for solving monotone inclusion problems involving the sum of three maximal monotone operators in a real Hilbert space $\HH$. The general formulation of the {\em inclusion problem} is as follows:
\begin{equation}\label{A+B}
\mbox{Find} \ \quad x\in\HH  \ \quad  \mbox{such that} \ \ 0\in (A+B)x,
\end{equation}
where $A:\HH\to\HH$ (single-valued) and $B:\dom(B)\subseteq\HH\rightrightarrows\HH$ (set-valued) are maximal monotone operators. Throughout this paper we assume that the solution set of this problem, $\zer(A+B)$, is non-empty. Problem \eqref{A+B} appears in different fields of applied mathematics
and optimization including signal processing, numerous important structured optimization, composite convex optimization, saddle point, and inverse problems; see, for instance,  \cite{IApp,kor}. 
One of the most  relevant setting that can be represented by a particular case of the inclusion problem \eqref{A+B} is the broadly-studied {\em variational inequality problem} (VIP)  \begin{equation}\label{VIP}\la Ax, y-x \ra\ge 0 \quad\text{for all}\quad y\in C.\end{equation} The set $C$ is a convex and closed subset of $\HH$ and $\HH$ is equipped with the inner product  $\la\cdot,\cdot\ra$ and the induced Euclidean norm $\|\cdot\|$.  This problem is a particular case of problem  \eqref{A+B} by taking $B=\mathcal{N}_C$ the normal cone of $C$, i.e., find $x\in C$ such that $0\in Ax+\mathcal{N}_C(x)$. A popular strategy to solve problem \eqref{VIP} is the so-called cutting plane (a.k.a. localization) idea which is based on finding a suitable hyperplane that separates the solution of the problem from the current iterate and then performs a metric projection step. This kind of idea is used by the famous {\em Extragradient method} and its variants for solving problem \eqref{VIP}; see, for instance, \cite{BI,kor,pang}. 

Here we apply the cutting plane idea to perform the first phase of the proposed iteration, described below when the single value operator $A$ inside of problem \eqref{A+B} is the sum of two parts. The considered iteration solves problem \eqref{A+B} when $A=A_1+A_2$ such that $A_1$ is cocoercive and $A_2$ is maximally monotone. Moreover, it uses a novel backtracking procedure that allows larger stepsizes in general. Furthermore, the {\em forward} steps are special projection steps onto suitable halfspaces which only evaluate $A_2$. It is worth noting that this kind of modification was presented in \cite{yun-reinier-1} to improve Tseng's scheme in finite dimension for solving problem \eqref{A+B} (without considering the cocoercive part). In general, the scheme in  \cite{yun-reinier-1} fails to keep the splitting structure of Tseng's splitting method and requires finding a uniformly bounded sequence in the image of the set-valued operator.  A similar approach (using normals vectors) for solving problem \eqref{VIP} was presented in \cite{yuni-rein-phan}.

\subsection{Splitting Iterations Description}

We focus our attention on a class of schemes, called splitting methods, which only use each operator individually rather than evaluating their sum directly. We refer to {\em forward step} when the single-valued operator is evaluated, and {\em backward step} when the resolvent operator of the set-valued operator is computed. Recalling that the {\em resolvent operator} of $B$ is the full domain single-valued operator in $\HH$ given by $J_{B}:=(I+B)^{-1}$ where $I\colon\HH\to\HH$ denotes the identity operator.

One of the most important classical splitting methods to find a zero of the sum $A+B$ is the so-called {\em forward-backward} (FB) splitting method introduced in \cite{D-R} which is given as follows:
\begin{equation}\label{F-B}x^{k+1}:=J_{\alpha_k B}(x^k-\alpha_k Ax^k), \end{equation}
where $\alpha_k >0$ for all $k\in\NN$. This iteration converges weakly to a point in the solution of problem \eqref{A+B}, $\zer(A+B)$, when the inverse of $A$ is $\beta$-strongly monotone (or equivalently $A$ being $\beta$-cocoercive), i.e., 
$$\forall x,y\in \HH,\qquad \langle Ax-Ay,x-y\rangle\ge \beta\|Ax-Ay\|^2,$$
where $\alpha_k \le \beta$ for all $k\in\NN$ and $\liminf_{k\to\infty}\alpha_k>0$; see, for instance, \cite{Lions1979,Passty}. Unfortunately, there is no chance to relax the cocoercivity assumption on $A$ to plain monotonicity and still prove convergence for the FB iteration given in \eqref{F-B}. For example, if we set $A$ as the  $\pi/2$-rotation operator which is monotone and $B=0$, iteration \eqref{F-B} moves away from zero (the unique solution) for any positive stepsize and starting at any point. Moreover, iteration \eqref{F-B} converges only weakly and the strong convergence could fail in general; see \cite{guler}. It is worth emphasizing that the cocoercivity assumption of an operator is a strictly stronger property than Lipschitz continuity. Recalling that, for some $L\geq 0$, $A$ is $L$-Lipschitz if $$\forall x,y\in\HH,\qquad\|Ax-Ay\|\le L\|x-y\|.$$ Note that $\beta$-cocoercive operators are monotone and $1/\beta$-Lipschitz continuous, but the converse does not hold in general, i.e., There exist monotone and Lipschitz continuous operators that are not cocoercive.  
Although, for gradients of lower semicontinuous, proper and convex functions, the cocoercivity is equivalent to the global Lipschitz continuity assumption. This nice and surprising fact is strongly used in the convergence analysis of the FB iteration \eqref{F-B} for solving the sum of two convex function (problem \eqref{A+B} with $A=\nabla f$ and $B=\partial g$), is known as
the {\em Baillon-Haddad Theorem}; see Corollary $18.16$ of \cite{HP}. Another useful feature that {\em Baillon-Haddad Theorem} assures is that $\nabla f$ is firmly nonexpansive if and only if it is a nonexpansive map. Recalling that an operator $A$ is non expansive if it is Lipschitz with constant $1$, and $A$ is said to be firmly nonexpansiveness if $$\forall x,y\in\HH,\qquad\|Ax-Ay\|^2\le \|x-y\|^2-\|(x-Ax)-(y-Ay)\|^2.$$ 
In order to relax the cocoercivity assumption, Tseng \cite{tseng} proposed a modification of the FB splitting method, known as the {\em forward-backward-forward} (FBF) splitting method, which usually requires the  $L$-Lipschitz continuity assumption of $A$ and an additional forward step. The FBF splitting iteration is:
\begin{align}
\bar x^k:=&J_{\alpha_k B}(x^k-\alpha_k Ax^k)\label{eqj}\\ x^{k+1}:=&\bar x^k-\alpha_k\big[A\bar x^k-Ax^k\big].\label{eq3}
\end{align} This iteration converges weakly to a point in $\zer(A+B)$, if: \begin{description}\item[(i)] the operator $A$ is monotone and $L$-Lipschitz and $L$ is available by taking $\alpha_k\le 1/L$ for all $k\in\NN$ and $\liminf_{k\to\infty}\alpha_k>0$; or \item[(ii)] the operator $A$ is locally uniformly continuous on $\dom(B)$ and the function $x \mapsto \min_{w\in (A+B)x}\|w\|$ is locally bounded on $\dom(B)$ by choosing $\alpha_k$ to be the largest $\alpha\in\{\sigma,\sigma\theta, \sigma\theta^2,\ldots\}$ with $\sigma>0$ and $\theta,\delta\in]0,1[$ satisfying 
\begin{equation}\label{TsengL}\alpha\|A\bar x^k-Ax^k\|\le \delta\|x^k-\bar x^k\|.\end{equation} \end{description}It is worth noting that there are relatively few effective alternatives to Tseng's FBF algorithm \eqref{eqj}-\eqref{eq3} for solving inclusions in the form of problem \eqref{A+B} \cite{Lions1979,Zhu-Marcotte-SIAM-1996,pontus}.

In this paper, we assume that the single-valued operator $A$ can be split as the sum of $A_1$ ($\beta$-cocoercive operator) and $A_2$ (monotone operator). Hence, problem \eqref{A+B} takes the following form:
\begin{equation}\label{A1+A2+B}
\mbox{Find} \ \quad x\in\HH  \ \quad  \mbox{such that} \ \ 0\in (A_1+A_2+B)x.
\end{equation}
For convenience,  we also denote $\zer(A_1+A_2+B)$ the solution set of this problem, which from now on is assumed to be nonempty.   Moreover, the FBF iteration \eqref{eqj}-\eqref{eq3} can be modified to the following {\em forward-backward-half-forward} (FBHF) splitting iteration proposed in \cite{LuisDavis} as follows:
\begin{align}
\bar x^k:=&J_{\alpha_k B}(x^k-\alpha_k Ax^k)\label{eqjh}\\ x^{k+1}:=&\bar x^k-\alpha_k\big[A_2\bar x^k-A_2x^k\big].\label{eq3h}
\end{align} The weak convergence occurs  when: 
\begin{description}\item[(i)]  the operator $A_1$ is $\beta$-cocoercive, $A_2$ is $L$-Lipschitz and $\beta$ and $L$ are available by taking  $\alpha_k \in \left]\eta, \min\{\beta,1/(2L)\}\right[$ with $\eta>0$ for all $k\in\NN$; or\item[(ii)] the operator $A_1$ is $\beta$-cocoercive and $\beta$ is available and $A_2$ is uniformly continuous in weakly compact subset of $\dom (B)$  by choosing $\alpha_k$ as the largest $\alpha\in \{2\beta\epsilon\theta,2\beta\epsilon\theta^2,\ldots\}$ with $\epsilon,\theta\in]0,1[$  satisfying \eqref{TsengL} with $\delta\in]0,\sqrt{1-\epsilon}[$ and  $A=A_2$.\end{description}
It is worth mentioning that this last backtracking strategy, described in (ii), to find $\alpha_k$ uses $\beta$ and allows smaller stepsizes than $\beta$. The main difference between iterations \eqref{eqj}-\eqref{eq3} and \eqref{eqjh}-\eqref{eq3h} yields in the last forward steps \eqref{eq3} and \eqref{eq3h}, i.e., the operator $A_1$ is not evaluated in \eqref{eq3h} (remind that $A=A_1+A_2$). This is possible because of the cocoercive assumption, which allows the one-step FB splitting iteration \eqref{F-B} to be convergent for certain values of $\alpha_k$ related to $\beta$. Actually, if $A_2=0$, the FB splitting iteration \eqref{F-B} is recovered by the FBHF splitting iteration \eqref{eqjh}-\eqref{eq3h}. In this context, with $A_1=\nabla f$ and $B=\partial g$ with $f,g$ convex functions, was proposed by using the backtracking procedure \eqref{TsengL} a weakly convergent proximal gradient method in \cite{yun-nghia} without any boundedness of the image of $\partial g$ or Lipschitz continuity assumption of $\nabla f$. Note further that, in the particular case that $A_1=0$, the problem \eqref{A1+A2+B} (with $A=A_2$) becomes problem \eqref{A+B} and the FBHF iteration \eqref{eqjh}-\eqref{eq3h} coincides with the FBF iteration \eqref{eqj}-\eqref{eq3}.
Motivations and applications for such kind of splitting structure giving in problem \eqref{A1+A2+B} can be found in \cite{LuisDavis,HP}. For example, Brice\~no-Arias and Davis in \cite{LuisDavis} applied the algorithm to primal-dual composite monotone inclusions with non-self-adjoint linear operators. In nonsmooth empirical risk minimizations in machine learning, one minimizes finite approximations of expected value for the loss function and constraints. For more, we refer the reader to \cite{LuisDavis} and the references therein. 

The proposed conceptual algorithm here modifies and extends the FBHF iteration  \eqref{eqjh}-\eqref{eq3h} by using two phases: (I) in the spirit of cutting plane methods, a backtracking search is performed to construct a suitable separating hyperplane; and (II) two special projection (forward) steps onto suitable separating hyperplanes deliver two different methods. Convergence analysis of both methods is
presented without the Lipschitz continuity assumption and the knowledge of the cocoercive constant of $A_1$. In addition, if $A_2$ is $L_2$-Lipschitz, the proposed backtracking strategy may allow larger stepsizes than $1/L_2$. Furthermore, the first variant relies on a very general iteration which recovers the FBHF iteration  \eqref{eqjh}-\eqref{eq3h} as a special case.  
The second variant has the following
desirable properties: (i) the generated sequence is entirely contained in a ball with a diameter equal to the distance between the initial and the solution set; and (ii) the whole sequence converges strongly to the solution of the problem that lies closest
to the initial point. Emphasizing that only weak convergence is known for the FBHF splitting method.

The presentation of this paper is as follows. In the next section, we provide some relevant background and useful facts that will be used throughout this paper. The proposed conceptual algorithm is presented in Section \ref{section3} and its two versions, called {\bf Method 1} and {\bf Method 2} are described. Section \ref{section4} contains the convergence analysis of these methods. Section \ref{section5} gives some concluding remarks.

\section{Preliminaries}
In this section, we present some definitions and conventional results needed for the convergence analysis of the proposed methods. The notation and results we discuss are standard, and interested readers can find further information in \cite{HP}. 

Throughout this paper, we write $p:=q$ to indicate that $p$ is defined to be equal to $q$. Let $\HH$ be a real Hilbert space equipped with inner product $\la \cdot , \cdot \ra$ and induced norm $\|\cdot\|:=\sqrt{\la\cdot,\cdot\ra}$. We
write $\NN$ for the nonnegative integers $\{0, 1, 2,\ldots\}$. The closed ball centered at $x\in\HH$ with radius $\gamma>0$ will be denoted by $\mathbb{B}[x;\gamma]:=\menge{y\in\HH }{ \|y-x\|\leq\gamma}$. 
Let $T:\HH\rightrightarrows\HH$ be a set-valued operator and its domain be denoted by $\dom(T):=\menge{x\in\HH}{T(x)\neq \emp}$ and, 
for simplicity, we usually write $Tx := T(x)$.  Define the graph of $T$ by $\gr(T):=\menge{(x,u)\in\dom(T)\times\HH}{u\in Tx}$. 
We say that $T$ is monotone if 
\begin{equation*}\label{oioi}\forall (x,u), (y,v)\in\gr(T),
\qquad\scal{x-y}{u-v}\geq 0,
\end{equation*}
and it is maximally monotone if there exists no monotone operator $T^{\prime}$ such that $\gr(T^{\prime})$ properly contains $\gr(T)$.

In the following, we state some important facts and preliminary results on set-valued mappings that are maximally monotone and addresses their graphs properties. 
\begin{lemma}[Proposition 20.31 and Proposition 20.33 of \cite{HP}] \label{bound}
Let $T:\HH\rightrightarrows\HH$ be a maximal monotone operator and let $x\in\HH$. Then the following hold: 
\item[ {\bf(i)}] $Tx$ is closed and convex;

\item[ {\bf(ii)}] For every sequence $(x^k,u^k)_{k\in\NN}\subset \gr(T)$ and every point $(x,u)\in \dom(T)\times\HH$, if $x^k\rightharpoonup x$ and $u^k\to u$, then $(x,u)\in \gr(T)$, i.e. $\gr(T)$ is sequentially closed in the weak-strong topology;
\item[ {\bf(iii)}] $\gr(T)$ is closed in $\HH\times\HH$ in the strong topology.
\end{lemma}
Note that the graph of a maximal monotone operator, in general, need not be sequentially closed in the weak topology of  $\HH\times\HH.$
\begin{proposition}[Theorem $4$ of \cite{minty}]\label{inversa}
Let $T:\dom(T)\subseteq\HH\rightrightarrows\HH$ be a set-valued and maximal monotone operator. If $\alpha >0$ then the resolvent operator $J_{\alpha T}:=(I+\alpha\, T)^{-1}:\HH\rightarrow \dom(T)$ is a full domain, single-valued and firmly nonexpansive operator, i.e.,  $$\forall x,y\in\HH, \qquad \| J_{\alpha T}(x)-J_{\alpha T}(y)\|^2+\|(I-J_{\alpha T})(x)-(I-J_{\alpha T})(y)\|^2\le\| x-y\|^2.$$
\end{proposition}
The inverse of $T$ is the set-valued operator defined by $T^{-1}\colon u \mapsto \menge{x\in\HH}{ u\in T(x)}$.  The zero set of $T$ is $\zer(T) := T^{-1}(0)$. If $T=A+B$ then the solution of problem \eqref{A+B} is $$\zer(A+B)=(A+B)^{-1}(0)=\menge{x\in\HH}{0\in (A+B)x}.$$ The next result characterizes the above solution set as the fixed points of the forward-backward operator.  
\begin{proposition}[Proposition 23.28 of \cite{HP}]\label{parada}
Let $\alpha>0$, and $A:\HH\to\HH$ and $B:\dom(B)\subseteq\HH\rightrightarrows\HH$ be two maximal monotone operators. Then, $$x=(I+\alpha B)^{-1}(x-\alpha Ax)\quad\text{if and only if}\quad x\in \zer(A+B).$$
\end{proposition}
Note further that, for all $x,y\in\HH$,  \begin{equation}\label{fermat}y=(I+\alpha B)^{-1}(x-\alpha Ax)\; \quad\text{if and only if}\quad  \frac{x-y}{\alpha}-Ax\in By.\end{equation} 
Let $C$ be a nonempty, convex and closed subset of $\HH$, and define the normal cone operator with respect to a nonempty closed convex set $C\subseteq \HH$ as
\[
\mathcal{N}_C(x):=\left\{ \begin{array}{ll}
\emp, & \text{if}\; x \not\in C\\
\menge{ y\in\HH}{\scal{y}{z-x}\leq 0,~~\forall z\in C}, &  \text{if}\;  x \in C.
\end{array}\right.
\] Hence, the orthogonal projection of $x$ onto $C$, $\Pi_C(x)$, is given by $\Pi_C(x)=J_{\mathcal{N}_C}(x)=(I+\mathcal{N}_C)^{-1}(x)$. 
Now, we state two well-known facts on orthogonal projections.
\begin{proposition}[Theorem 3.16 and Proposition 4.16 of \cite{HP}]\label{proj}
Let $C$ be nonempty closed convex subset of $\HH$, and $\Pi_C$ be the orthogonal projection onto $C$. For all $x,y\in \HH$ and all $z\in C$ the following hold:
\item[ {\bf(i)}] $ \|\Pi_{C}(x)-\Pi_{C}(y)\|^2 \leq \|x-y\|^2-\|(x-\Pi_{C}(x))-\big(y-\Pi_{C}(y)\big)\|^2;$
\item[ {\bf(ii)}] $\la x-\Pi_C(x),z-\Pi_C(x)\ra \leq 0.$
\end{proposition}
In the following, we present some useful formulae to describe the iterates of the proposed methods by introducing suitable hyperplanes and orthogonal projections onto these hyperplanes. 
\begin{proposition}[Proposition 28.19 of \cite{HP}]\label{proj_C(x)}
Let $y,v\in\HH$, $r\in\RR$,
$$\Tau^r_{y,v}:=\menge{x\in\HH}{\la v,x-y\ra\leq r},$$ 
and $$\Gamma_{z,x^0}:=\menge{x\in\HH}{\la x^0-z,x-z\ra\leq 0}.$$ 
Then, 
$$\Pi_{\Tau^r_{y,v}}(w)=\left\{\begin{array}{lll}w, &\mbox{if} & w\in\Tau^r_{y,v} \\ w-\displaystyle\frac{\la v,w-y\ra-r\,}{\|v\|^2}\,v,&\mbox{if} & w\notin \Tau^r_{y,v}. \end{array}\right.$$
Moreover, 
\item [ {\bf (i)}] if $x^0-z$ is linearly dependent to $v$ (or equivalently, $\|x^0-z\|\|v\|=\langle x^0-z,v\rangle$), $\Tau^r_{y,v}\subset\Gamma_{z,x^0}$ and $$\Pi_{\Tau^r_{y,v}\cap\Gamma_{z,x^0}}(x^0)=\Pi_{\Tau^r_{y,v}}(x^0).$$
\item [ {\bf (ii)}] if $x^0-z$ is linearly independent to $v$ (or equivalently, $\|x^0-z\|\|v\|>\langle x^0-z,v\rangle$),
$$\Pi_{\Tau^r_{y,v}\cap\Gamma_{z,x^0}}(x^0)=x^0-\lambda_1v-\lambda_2(x^0-z),$$  
where $\lambda_1,\lambda_2$ are solutions of the linear system:
\begin{align*}\lambda_1\|v\|^2+\lambda_2\la v,x^0-z\ra&=\la v,x^0-y\ra-r\\\lambda_1\la v, x^0-z\ra+\lambda_2\|x^0-z\|^2&=\la x^0-z,x^0-z\ra.\end{align*}
\end{proposition}
Now we define an important concept, the so-called Fej\'er monotonicity. 
\begin{definition}
Let $S$ be a nonempty subset of $\HH$. A sequence $(x^k)_{k\in\NN}\subset\HH$ is said to be Fej\'er monotone with respect to $S$, if and only if, for all $x\in S$ there exists $k_0\in\NN$, such that $$\|x^{k+1}-x\|\le\|x^k-x\|\quad \text{for all}\quad k\ge k_0.$$
\end{definition}
Useful properties on Fej\'er monotone sequences are the following.
\begin{proposition}[Proposition 5.4 and Theorem 5.5 of \cite{HP}]\label{punto}
Let $(x^k)_{k\in\NN}$ be a sequence in $\HH$ and let $S$ be a non empty subset of $\HH$. If $(x^k)_{k\in\NN}$ is Fej\'er monotone with respect to $S$, then:
\item[ {\bf(i)}] The sequence $(x^k)_{k\in\NN}$ is bounded;
\item[ {\bf(ii)}] The sequence $\big(\|x^k-x\|\big)_{k\in\NN}$ is convergent for all $x\in S;$
\item[ {\bf(iii)}] If every weak accumulation point $x^*$ of $(x^k)_{k\in\NN}$ belongs to $S$, then $(x^k)_{k\in\NN}$ converges weakly to $x^*$.
\end{proposition}

\section{The Conceptual Forward-Backward-Half-Forward Algorithm}\label{section3}

The conceptual modification of the FBHF splitting algorithm uses the parameters $\theta, \delta\in(0,1)$. It is defined as follows:
 \vspace{-0.25in}
\begin{center}\fbox{\begin{minipage}[b]{\textwidth}

\noindent{\bf Conceptual Algorithm.} 

\medskip

\noindent{\bf Step 0. (Initialization):}
Take
\begin{equation*}\label{algoritmo_1_paso_0}
x^0\in\HH,\quad \mbox{and} \quad\alpha_{-1}>0.
\end{equation*}
\smallskip

\noindent {\bf Step 1. (Backtracking):} Given $x^k$ and $\alpha_{k-1}$ define
\begin{equation}{\label{jota}}
\bar{x}_j^k:=J_{\alpha_{k-1}\theta^jB}(x^{k}-\alpha_{k-1}\theta^jAx^{k}).
\end{equation}
Start the inner loop in $j$ to compute $j(k)$ as the smallest $j\in\NN$ such that 
\begin{equation}\label{linesearch}
\alpha_k\big\langle A_2x^k-A_2\bar{x}_j^k, x^k-\bar{x}_j^k\big\rangle\leq\delta\|x^k-\bar{x}_j^k\|^2. 
\end{equation}

\smallskip

\noindent {\bf Step 2. (Forward step):} Set
$\alpha_k:=\alpha_{k-1}\theta^{j(k)},$
\begin{equation}{\label{xbar}}
\bar{x}^k:=\bar x^k_{j(k)}=J_{\alpha_{k}B}(x^{k}-\alpha_{k}Ax^{k})
\end{equation}
and
\begin{equation}{\label{Fk}}
x^{k+1}:=\mathcal{F}(x^k,\bar x^k).
\end{equation}

\noindent {\bf Stopping Criterion:} If $x^{k+1}=x^k$ then stop.
\end{minipage}}\end{center}
We consider two projection variants of the {\bf Conceptual Algorithm}, which are called  {\bf Method 1} and {\bf Method 2} respectively. It will be used two different forward steps $\mathcal{F}=\mathcal{F}_1$ and $\mathcal{F}=\mathcal{F}_2$ on the projection steps in \eqref{Fk} as follows. Take any $\bar \delta$ such that $1-\delta-\bar \delta>0$ and define $r_k:=\frac{\bar \delta}{\alpha_k}\|x^k-\bar{x}^k\|^2$,
\begin{equation}\label{H(x)}
\Tau_k:=\left\{x\in\HH \,\left|\,\left\langle \frac{x^k-\bar{x}^k}{\alpha_k}-(A_2x^k-A_2\bar{x}^k),x-\bar{x}^k\right\rangle\le r_k\right.\right\}
\end{equation}
and
\begin{equation}
  \mathcal{F}_1(x^k,\bar x^k):=\Pi_{\Tau_k}(x^k).\label{iterado2}
\end{equation}
Moreover, set
\begin{equation}\label{W(x)}
\Gamma_k:=\menge{x\in\HH}{\la x^0-x^k,x-x^k\ra\le 0}
\end{equation}
and
\begin{equation}
  \mathcal{F}_2(x^k,\bar x^k):=\Pi_{\Tau_k\cap \Gamma_k}(x^0)\label{iterado3}.
\end{equation}
Similar forward steps have been used in several papers for solving nonsmooth convex optimization problems \cite{yuniusem},  variational inequalities \cite{yuni-rein-phan,BI}, and nonsmooth inclusion problems \cite{sva, CB}. The existence of $j(k)$, satisfying \eqref{linesearch}, and the well-definition of \eqref{iterado2}  and \eqref{iterado3} will be proved in the next section. It is worth mentioning that the projection steps defined in \eqref{iterado2} and \eqref{iterado3} do not increase the computational burden per iteration, that is, both steps have closed and inexpensive formulae. By using Proposition \ref{proj_C(x)} with $y=\bar x^k$, $z=x^k$, $$u=r^k:=(\bar \delta/\alpha_k) \|x^k-\bar x^k\|^2$$ and 
$$v=\bar{w}_2^k:=\frac{x^k-\bar{x}^k}{\alpha_k}-(A_2x^k-A_2\bar{x}^k),$$ we have $\Tau_k=\Tau^{r_k}_{\bar x^k,\bar w^k}$ and $\Gamma_k=\Gamma_{x^k,x^0}$. Moreover, since $x^k\notin \Tau_k$ which is proved below in Proposition \ref{H<=>S^*}, $$\mathcal{F}_1(x^k,\bar x^k)= \Pi_{\Tau_k}(x^k)= x^k-\frac{\la\bar{w}_2^k,x^k-\bar{x}^k\ra-r_k\,}{\|\bar{w}_2^k\|^2}\,\bar{w}_2^k$$ and also $$\mathcal{F}_2(x^k,\bar x^k)=\Pi_{\Tau_k\cap\Gamma_k} (x^0)=x^0-\lambda^k_1\bar{w}_2^k-\lambda^k_2(x^0-x^k),$$ where $\lambda^k_1,\lambda^k_2$ are given by, 
\begin{equation}\label{l1}\lambda^k_1=\left\{\!\begin{array}{lll}\displaystyle\frac{\left(\la\bar{w}_2^k,x^0-\bar{x}^k\ra-r_k\right)\left\|x^0-x^k\right\|^2-\la\bar{w}_2^k,x^0-x^k\ra\left\|x^0-x^k\right\|^2}{\left\|\bar{w}_2^k\right\|^2\left\|x^0-x^k\right\|^2-\la\bar{w}_2^k,x^0-x^k\ra^2},\!&\mbox{if}&\!\!\!\displaystyle\frac{\langle\bar{w}_2^k,x^0-x^k\rangle}{\|\bar{w}_2^k\|\|x^0-x^k\|}<1\\
\\\displaystyle\frac{\la\bar{w}_2^k,x^0-\bar{x}^k\ra-r_k\,}{\|\bar{w}_2^k\|^2},\!&\mbox{if}&\!\!\!\displaystyle\frac{\langle\bar{w}_2^k,x^0-x^k\rangle}{\|\bar{w}_2^k\|\|x^0-x^k\|}=1\end{array}\right.\end{equation} and
\begin{equation}\label{l2}\lambda^k_2=\left\{\begin{array}{lll}\displaystyle\frac{\left\|\bar{w}_2^k\right\|^2\left\|x^0-x^k\right\|^2-\la\bar{w}_2^k,x^0-x^k\ra\left(\la\bar{w}_2^k,x^0-\bar{x}^k\ra-r_k\right)}{\left\|\bar{w}_2^k\right\|^2\left\|x^0-x^k\right\|^2-\la\bar{w}_2^k,x^0-x^k\ra^2},&\mbox{if}&\displaystyle\frac{\langle\bar{w}_2^k,x^0-x^k\rangle}{\|\bar{w}_2^k\|\|x^0-x^k\|}<1\\0,&\mbox{if}&\displaystyle\frac{\langle\bar{w}_2^k,x^0-x^k\rangle}{\|\bar{w}_2^k\|\|x^0-x^k\|}=1.\end{array}\right.\end{equation}

\section{Convergence Analysis}\label{section4}
We start this section by presenting some technical results that are useful in analyzing the convergence properties of the two proposed methods. We also prove that {\bf Conceptual Algorithm} is well-defined. We start proving that \eqref{linesearch} is satisfied by $j$ sufficiently large, hence $\alpha_k$ is well defined. From now on, we assume that $A_2:\HH\to\HH$ is a uniformly continuous mapping. This assumption is standard to prove weak convergence of the FBHF splitting method without the Lipschitz continuity assumption. Moreover, we assume that the solution set of the inclusion problem \eqref{A1+A2+B}, $\zer(A_1+A_2+B)$, is nonempty. 

\begin{proposition}
	\label{unifm-cont-welldefined}
The inequality \eqref{linesearch} in the backtracking strategy holds after finitely many steps.
\end{proposition}
\begin{proof}
If $x^k\in \zer(A_1+A_2+B)$ then \eqref{linesearch} automatically holds because Proposition \ref{parada}. Thus, assume that $x^k$ is not a solution, i.e., $x^k\notin \zer(A_1+A_2+B)$. So, using \eqref{jota} and Proposition \ref{parada}, we have $x^k\neq \bar x_j^k=(I+\alpha_{k-1}\theta^jB)^{-1}(x^{k}-\alpha_{k-1}\theta^jAx^{k})$ for all $j\in\NN$. The proof of the well-definition of $j(k)$ goes by contradiction. Assume that \eqref{linesearch} does not hold, i.e.,  
$$\delta\|x^k-\bar{x}^k_j\|^2 < \alpha_{k-1}\theta^j\big\langle A_2x^k-A_2\bar{x}^k_j,x^k-\bar{x}^{k}_{j}\big\rangle\le \alpha_{k-1}\theta^j \|A_2x^k-A_2\bar{x}^k_j\|\|x^k-\bar{x}^{k}_{j}\|,$$ using  the Cauchy-Schwartz inequality in the last inequality.
Dividing by $\|x^k-\bar{x}^k_j\|\neq 0$, we get 
\begin{align*}
\delta\|x^k-\bar{x}^k_j\| < \alpha_{k-1}\theta^j \|A_2x^k-A_2\bar{x}^{k}_{j}\|.
\end{align*}
Since $\theta\in (0,1)$ and $\|A_2x^k-A_2\bar{x}^{k}_{j}\|$ is bounded for all $j\in\NN$, then by letting $j\to +\infty$ the right hand side of the last inequality goes to zero. Hence $\|x^k-\bar{x}^k_j\|\to 0$. Since $A$ is uniformly continuous, we have 
\begin{equation}\label{Akto0}\|A_2x^k-A_2\bar{x}^k_j\|\to 0.\end{equation} 
Consequently, \begin{equation}\label{xk-barxkto0}\displaystyle\frac{\|x^k-\bar{x}^k_j\|}{\alpha_{k-1}\theta^j}\to 0.\end{equation} 
Moreover, the $\beta$-cocoersivity of $A_1$ implies $1/\beta$-Lipschitz continuity. Then, \begin{equation*}\lim_{j\to\infty}\|A_1x^k-A_1\bar{x}^k_j\|\le 1/\beta\lim_{j\to\infty}\|x^k-\bar{x}^k_j\|=0, \end{equation*} which implies \begin{equation}\label{A22kto0}\|A_1x^k-A_1\bar{x}^k_j\|\to 0.\end{equation} 
Define, $$\bar{w}^k_j:=\displaystyle\frac{x^k-\bar{x}^k_j}{\alpha_{k-1}\theta^j}-(Ax^k-A\bar{x}^k_j).$$ It follows from \eqref{fermat} that $\bar{w}^k_j\in(A+B)\bar{x}^k_j$, or equivalently, $$(\bar{x}^k_j,\bar{w}^k_j)\in \gr(A+B).$$ Observe that $A=A_1+A_2$ and by the Cauchy-Schwarz inequality, we get
\begin{align*}\|\bar{w}^k_j\|&=\left\|\displaystyle\frac{x^k-\bar{x}^k_j}{\alpha_{k-1}\theta^j}-\left[(A_1+A_2)x^k-(A_1+A_2)\bar{x}^k_j\right]\right\|\\&\le \displaystyle\frac{\|x^k-\bar{x}^k_j\|}{\alpha_{k-1}\theta^j}+\|A_1x^k-A_1\bar{x}^k_j\|+\|A_2x^k-A_2\bar{x}^k_j\|.\end{align*}
Hence,   $\bar{w}^k_j$ converges to $0$ by using \eqref{Akto0}, \eqref{xk-barxkto0} and \eqref{A22kto0} above. Since $\bar{x}^k_j\rightharpoonup x^k$ and $\bar{w}^k_j\to 0$ and by Lemma \ref{bound}(ii), $\gr (A+B)$ is closed in the weak-strong topology. So,  $$(x^k,0)\in \gr(A+B),$$ or  equivalently,
$0\in(A+B)x^k=(A_1+A_2+B)x^k$. Therefore, $x^k\in \zer(A_1+A_2+B)$ which is a contradiction. 
\end{proof}
\begin{remark} 
We notice that if $A_2$ is $L_2$-Lipschitz continuous then any $\alpha\leq\displaystyle\frac{\delta}{L_2}$ satisfies the backtracking inequality \eqref{linesearch}. Actually, if we use the Cauchy-Schwartz inequality in the left hand side of \eqref{linesearch}, we get $$\alpha\la A_2x^k-A_2\bar{x}^k,x^k-\bar{x}^k\ra\leq\frac{\delta}{L_2}\|A_2x^k-A_2\bar{x}^k\|\|x^k-\bar{x}^k\|\leq\delta\|x^k-\bar{x}^k\|^2.$$ Moreover, it is easy to prove that $(\alpha_k)_{k\in\NN}$ the sequence generated by the backtracking strategy given by \eqref{linesearch} satisfies $$\alpha_k\ge \min\left\{\alpha_{-1},\frac{\delta}{L_2}\right\}$$ for all $k\in\NN$. 
\end{remark}
From now on, we assume that $\alpha_{-1}\le 4\beta\bar \delta$ where $\beta$ is the cocoercive constant for $A_1$. 
\begin{lemma}
	\label{S*subsetH}
Let $(x^k)_{k\in\NN}$, $(\bar x^k)_{k\in\NN}$ and $(\alpha_k)_{k\in\NN}$ be the sequences generated by {\bf Conceptual Algorithm}. Then, for all $k\in\NN$, 
\item [ {\bf (i)}] $\displaystyle\frac{x^k-\bar{x}^k}{\alpha_k}-(A_2x^k-A_2\bar x^k)\in (A_2+B)\bar{x}^k+A_1x^k$; 

\item [ {\bf (ii)}] $\zer(A_1+A_2+B)\subseteq \Tau_k$.
\end{lemma}
\begin{proof} By the definition of $\bar x^k$ given in \eqref{xbar}  and \eqref{fermat},  $\displaystyle\frac{x^k-\bar{x}^k}{\alpha_k}-Ax^k\in B\bar{x}^k$. So, (i) follows after add $A_2\bar x^k+A_1x^k$ in both sides and use that $A=A_1+A_2$. To prove (ii) take any $x^*\in \zer(A_1+A_2+B)$. Then, there exists $v^{*}\in B(x^{*})$, such that $0=A_1x^{*}+A_2x^*+v^{*}$. Using (i) and also the monotonicity of $A_2+B$ give us
\begin{align*}0&\le\Big\langle\frac{x^k-\bar{x}^k}{\alpha_k}-(A_2x^k-A_2\bar{x}^k)-A_1x^k-(A_2x^*+v^*), \bar{x}^k-x^*\Big\rangle\\&=\Big\langle\frac{x^k-\bar{x}^k}{\alpha_k}-(A_2x^k-A_2\bar{x}^k)-A_1x^k+A_1x^*, \bar{x}^k-x^*\Big\rangle.\end{align*}
Rearranging, we have 
\begin{align*}\Big\langle\frac{x^k-\bar{x}^k}{\alpha_k}-(A_2x^k-A_2\bar{x}^k), x^*-\bar{x}^k\Big\rangle &\le\Big\langle A_1x^*-A_1x^k, \bar{x}^k-x^*\Big\rangle.\end{align*} 
Now, by using the $\beta$-cocoercivity of $A_1$, we get
\begin{align*}\langle A_1x^*-A_1x^k, \bar{x}^k-x^*\rangle&=\langle A_1x^*-A_1x^k, {x}^k-x^*\rangle+\langle A_1x^*-A_1x^k, \bar{x}^k-x^k\rangle\\&\le -\beta \|A_1x^*-A_1x^k\|^2+\frac{1}{2\alpha_k}\left[2\langle \alpha_kA_1x^*-\alpha_kA_1x^k, \bar{x}^k-x^k\rangle\right].\end{align*} So, using the identity, for any $\alpha,\gamma>0$ and $a,b\in\HH$,
$
2\langle \alpha b, a\rangle= \gamma\|a\|^2+\frac{\alpha^2}{\gamma}\|b\|^2-\gamma\|a-\frac{\alpha}{\gamma}b\|^2$ in the right hand side of the last inequality,
 we get
\begin{align*}\langle A_1x^*-A_1x^k, \bar{x}^k-x^*\rangle\le&\frac{\gamma}{2\alpha_k}\|\bar{x}^k-x^k\|^2+\left(\frac{\alpha_k}{2\gamma}-\beta\right)\|A_1x^*-A_1x^k\|^2\\&-\frac{\gamma}{2\alpha_k}\|\bar{x}^k-x^k-\frac{\alpha_k}{\gamma}(A_1x^*-A_1x^k)\|^2,\end{align*} for any $\gamma>0$.
Therefore, taking $\gamma=\frac{\alpha_{-1}}{2\beta}\ge \frac{\alpha_{k}}{2\beta}$ for all $k\in\NN$ and using that $\alpha_{-1}\le4\beta\bar \delta$, we have
$$\Big\langle\frac{x^k-\bar{x}^k}{\alpha_k}-(A_2x^k-A_2\bar{x}^k), x^*-\bar{x}^k\Big\rangle\le \frac{\bar \delta}{\alpha_k}\|x^k-\bar{x}^k\|^2=r_k$$ and by \eqref{H(x)}, $x^{*}\in \Tau_k$ as desired.
\end{proof}
\begin{proposition}\label{coro}
Let $(x^k)_{k\in\NN}$, $(\bar x^k)_{k\in\NN}$ and $(\alpha_k)_{k\in\NN}$ be the sequences generated by {\bf Conceptual Algorithm}. Then, for all $k\in\NN$, $\alpha_k\leq\alpha_{-1}$ and
\begin{equation}\label{desig-muy-usada}
\Big\langle\frac{x^k-\bar{x}^k}{\alpha_k}-(A_2x^k-A_2\bar{x}^k),x^{k}-\bar{x}^{k}\Big\rangle\geq\frac{1-\delta}{\alpha_{k}}\|x^{k}-\bar{x}^k\|^2\geq0.
\end{equation}
\end{proposition}
\begin{proof} The fact that $\alpha_k\leq\alpha_{-1}$ for all $k\in\NN$ follows from the definition of the backtracking strategy inside of {\bf Conceptual Algorithm}.
Using the line search inequality, we have 
\begin{align*}
\Big\langle\frac{x^k-\bar{x}^k}{\alpha_k}-(A_2x^k-A_2\bar{x}^k),x^{k}-\bar{x}^{k}\Big\rangle&=\frac{\|x^k-\bar{x}^k\|^2}{\alpha_k}-\langle A_2x^k-A_2\bar{x}^k,x^{k}-\bar{x}^{k}\rangle\\&\geq \frac{\|x^k-\bar{x}^k\|^2}{\alpha_{k}}-\frac{\delta}{\alpha_{k}}\|x^{k}-\bar{x}^{k}\|^2\\&= \frac{1-\delta}{\alpha_{k}}\|x^k-\bar{x}^k\|^2.
\end{align*}
So, the result follows.
\end{proof}
\begin{proposition}\label{H<=>S^*} Let $(x^k)_{k\in\NN}$, $(\bar x^k)_{k\in\NN}$ and $(\alpha_k)_{k\in\NN}$ be the sequences generated by {\bf Conceptual Algorithm}. Then, $x^k\in \Tau_k$ if and only if $x^k\in \zer(A_1+A_2+B)$.
\end{proposition}
\begin{proof}
Clearly, from Lemma \ref{S*subsetH}, if $x^k\in \zer(A_1+A_2+B)$ then $x^k\in\Tau_k$. Conversely, 
if $x^k\in \Tau_k$ then
\begin{align*}
\frac{\bar \delta}{\alpha_k}\|x^k-\bar{x}^k\|^2&\geq \Big\langle\frac{x^k-\bar{x}^k}{\alpha_k}-(A_2x^k-A_2\bar{x}^k), x^k-\bar{x}^k\Big\rangle \ge \frac{1-\delta}{\alpha_{k}}\|x^k-\bar{x}^k\|^2 ,
\end{align*} using  Proposition \ref{coro} in the second inequality.
Hence, 
$ \displaystyle\frac{(1-\delta-\bar \delta)}{\alpha_{k}}\|x^k-\bar{x}^k\|^2 \le 0$, 
which implies that
$x^k=\bar{x}^k$ and by Proposition \ref{parada}, $x^k\in \zer(A_1+A_2+B)$. 
\end{proof}
\subsection{Convergence of Method 1}
In this subsection all results are referred to {\bf Method 1}, i.e., with {\bf Conceptual Algorithm} with the iterative step \eqref{Fk} in {\bf Step 2.} as $$x^{k+1}=\mathcal{F}(x^k,\bar x^k)=\mathcal{F}_1(x^k,\bar x^k)=\Pi_{\Tau_k}(x^k).$$
or equivalently, 
\begin{equation}\label{meth1F}
x^{k+1}=\Pi_{\Tau_k}(x^k)=x^k-\frac{\langle \displaystyle \bar w^k_2,x^k-\bar{x}^k\rangle-\frac{\bar \delta}{\alpha_k}\|x^k-\bar{x}^k\|^2}{\|\displaystyle\bar w^k_2\|^2}  \,\bar w^k_2, \end{equation} reminding that $\bar w^k_2=\displaystyle\frac{x^k-\bar{x}^k}{\alpha_k}-(A_2x^k-A_2\bar{x}^k)$.
If we define \begin{equation}\label{beta-k}\lambda_k:=\frac{\langle \displaystyle \bar w^k_2,x^k-\bar{x}^k\rangle-\frac{\bar \delta}{\alpha_k}\|x^k-\bar{x}^k\|^2}{\|\displaystyle\bar w^k_2\|^2}\end{equation} then \eqref{meth1F} becomes 
\begin{equation}\label{giteration}
x^{k+1}=\left(1-\frac{\lambda_k}{\alpha_k}\right)x^k+\frac{\lambda_k}{\alpha_k}\bar{x}^k-\lambda_k(A_2\bar{x}^k-A_2x^k).
\end{equation} The above forward step is interesting on its own. It is possible to use \eqref{giteration} to allow over and under projections onto $\Tau_k$, i.e., changing $\lambda_k$ by  $\gamma \lambda_k$ with $\gamma\in (0,2)$. The analysis of convergence for the resulted new variants follows the same lines of convergence of {\bf Method 1}.  Note further that if $\lambda_k=\alpha_k$, the forward step of the FBHF splitting iteration is recovered from \eqref{giteration}. 
\begin{proposition}\label{stop1}
If {\bf Method 1} stops then $x^k\in \zer(A_1+A_2+B)$.
\end{proposition}
\begin{proof}
If Stopping Criterion is satisfied, then  $x^{k+1}=\Pi_{\Tau_k}(x^k)=x^k$, which implies that $x^k\in \Tau_k$ and by Proposition \ref{H<=>S^*}, $x^k\in \zer(A_1+A_2+B)$. 
\end{proof}

From now on, we assume that {\bf Method 1} does not stop. Note that $\Tau_k$ is nonempty for all $k\in\NN$ by Lemma \ref{H<=>S^*}. Then, the projection step \eqref{iterado2} is well-defined, i.e., if {\bf Method 1} does not stop, it generates an infinite sequence $(x^k)_{k\in\NN}$.
\begin{proposition}\label{prop2}
Let $(x^k)_{k\in\NN}$ be a generated sequence by {\bf Method 1}. Then the following items are satisfied:
\item[{\bf(i)}] The sequence $(x^k)_{k\in\NN}$ is Fej\'er monotone with respect to $\zer(A_1+A_2+B)$;
\item[{\bf(ii)}] The sequence $(x^k)_{k\in\NN}$ is bounded;
\item[{\bf(iii)}] $\lim_{k\to \infty}\left\langle\displaystyle x^k-\bar{x}^k -\alpha_k(A_2x^k-A_2\bar{x}^k),x^k-\bar{x}^k\right\rangle-\bar \delta\|x^k-\bar x^k\|^2=0$.
\end{proposition}
\begin{proof}
(i) Take $x^*\in \zer(A_1+A_2+B)$. Using \eqref{iterado2}, Proposition \ref{proj}(i) and Lemma \ref{H<=>S^*}, we have
\begin{align}\label{fejer-des}\nonumber\|x^{k+1}-x^{*}\|^2&=\|\Pi_{\Tau_k}(x^k)-\Pi_{\Tau_k}(x^{*})\|^2\\ &\leq\nonumber \|x^k-x^*\|^2-\|\Pi_{H_k}(x^k)-x^k\|^2\\ &\leq \|x^k-x^*\|. 
\end{align} 
Thus, the Fej\'er monotonicity applies.
(ii) Using proposition \ref{punto} then the sequence is bounded.
(iii) It follows from \eqref{meth1F} that
\begin{equation*}
x^{k+1}=\Pi_{\Tau_k}(x^k)=x^k-\frac{\langle \displaystyle \bar w^k_2,x^k-\bar{x}^k\rangle-\frac{\bar \delta}{\alpha_k}\|x^k-\bar{x}^k\|^2}{\|\displaystyle\bar w^k_2\|^2}  \,\bar w^k_2, \end{equation*}
and combining this with the second line of \eqref{fejer-des}, we have
\begin{align}\label{ineq}
\|x^{k+1}-x^{*}\|^2 &\leq \|x^k-x^*\|^2 -\left\|x^k-\frac{\langle \displaystyle \bar w^k_2,x^k-\bar{x}^k\rangle-\frac{\bar \delta}{\alpha_k}\|x^k-\bar{x}^k\|^2}{\|\displaystyle\bar w^k_2\|^2}  \,\bar w^k_2-x^k\right\|^2 \nonumber\\ &= \|x^k-x^*\|^2 -\frac{\Big(\langle \displaystyle \bar w^k_2,x^k-\bar{x}^k\rangle-\frac{\bar \delta}{\alpha_k}\|x^k-\bar{x}^k\|^2 \Big)^2}{\|\displaystyle\bar w_2^k\|^2}. \nonumber
\end{align}
Reordering,
\begin{equation*}
\frac{\Big(\langle \displaystyle \bar w^k_2,x^k-\bar{x}^k\rangle-\frac{\bar \delta}{\alpha_k}\|x^k-\bar{x}^k\|^2 \Big)^2}{\|\displaystyle\bar w_2^k\|^2}\le\|x^{k}-x^{*}\|^2-\|x^{k+1}-x^{*}\|^2.
\end{equation*}
It follows from Proposition \ref{punto}(ii) and the definition of $\bar w_2^k$ that 
\begin{align*}
0&=\lim_{k\to\infty}\frac{\left(\Big\langle\displaystyle\frac{x^k-\bar{x}^k}{\alpha_k}-(A_2x^k-A_2\bar{x}^k),x^k-\bar{x}^k\Big\rangle-\frac{\bar \delta}{\alpha_k}\|x^k-\bar{x}^k\|^2\right)^2}{\left\|\displaystyle\frac{x^k-\bar{x}^k}{\alpha_k}-(A_2x^k-A_2\bar{x}^k)\right\|^2}\\&=\lim_{k\to\infty}\frac{\left(\Big\langle x^k-\bar{x}^k-{\alpha_k}(A_2x^k-A_2\bar{x}^k),x^k-\bar{x}^k\Big\rangle-\bar\delta\|x^k-\bar x^k\|^2\right)^2}{\|\displaystyle x^k-\bar{x}^k-{\alpha_k}(Ax^k-A\bar{x}^k)\|^2}.
\end{align*}
The sequence $\left(\|\displaystyle x^k-\bar{x}^k-{\alpha_k}(A_2x^k-A_2\bar{x}^k)\|\right)_{k\in\NN}$ is bounded because the sequence $(x^k)_{k\in\NN}$ and $(\bar x^k)_{k\in\NN}$ are bounded and 
$
\|\displaystyle x^k-\bar{x}^k-{\alpha_k}(A_2x^k-A_2\bar{x}^k)\|\le \|\displaystyle x^k-\bar{x}^k\|+{\alpha_k}\|A_2x^k-A_2\bar{x}^k\|,
$
proving the desired result.
\end{proof}

Next we establish our main convergence result on {\bf Method 1}. 
\begin{theorem}\label{teo1*}
All weak accumulation points of $(x^k)_{k\in\NN}$ belong to $\zer(A_1+A_2+B)$.
\end{theorem}
\begin{proof}
Using Proposition \ref{prop2}(iii) and taking limits in \eqref{desig-muy-usada}, we have 
\begin{align*}\label{limite}
0&=\lim_{k\to \infty}\Big\langle x^{k}-\bar{x}^{k}-\alpha_{k}(A_2x^{k}-A_2\bar{x}^{k}),x^{k}-\bar{x}^{k}\Big\rangle-\bar\delta\|x^k-\bar x^k\|^2\\&\ge \lim_{k\to\infty}  (1-\delta-\bar\delta)\|x^{k}-\bar{x}^{k}\|^2\ge 0. \end{align*}
Implying that,
\begin{equation}
\lim_{k\to\infty}\|x^{k}-\bar{x}^{k}\|=0.
\label{limnormzero} 
\end{equation} 
Proposition \ref{prop2}(ii) guarantees the existence of weak accumulation points of $(x^k)_{k\in\NN}$.  Let $\hat{x}$ be any weak accumulation point of $(x^k)_{k\in\NN}$ and assume that $(x^{i_k})_{k\in\NN}$ is any subsequence of $(x^k)_{k\in\NN}$ that converges weakly to $\hat{x}$ and also, without loss of generality, assume that $\lim_{k\to\infty}\alpha_{i_k}=\bar{\alpha}$. 
Hence, it follows from \eqref{limnormzero} that $(\bar x^{i_k})_{k\in\NN}$ converges weakly to $\hat x$ as well.

In the following, we can consider two cases.

\noindent {\bf (a)} Assume that $\bar{\alpha}> 0$. Hence, \eqref{limnormzero} implies
\begin{equation}\label{limcero}
\lim_{k\to \infty}\frac{\|x^{i_k}-\bar{x}^{i_k}\|}{\alpha_{i_k}}\le\frac{1}{\bar \alpha}\cdot\lim_{k\to \infty}\|x^{i_k}-\bar{x}^{i_k}\|=0,
\end{equation} because the sequence $(\alpha_k)_{k\in\NN}$ is a nonincreasing sequence.  
Notice that, $$\bar{w}^{i_k}:=\frac{x^{i_k}-\bar{x}^{i_k}}{\alpha_{i_k}}-(A_1x^{i_k}-A_1\bar{x}^{i_k})-(A_2x^{i_k}-A_2\bar{x}^{i_k})\in (A_1+A_2+B)\bar{x}^{i_k},$$ which is equivalent to $(\bar{x}^{i_k},\bar{w}^{i_k})\in \gr(A_1+A_2+B)$. The fact that $\bar{w}^{i_k}\to 0$ follows by using Cauchy-Schwartz, \eqref{limcero}, \eqref{limnormzero}, the uniformly continuity assumption of $A_2$ and the cocoercivity of $A_1$. Since $\bar{w}^{i_k}\to 0$ and $\bar{x}^{i_k}\rightharpoonup\hat{x}$ then the closedness of $\gr (A_1+A_2+B)$ implies that $(\hat{x},0)\in \gr(A_1+A_2+B)$, and therefore $\hat{x}\in \zer(A_1
+A_2+B).$

\noindent {\bf (b)} Assume that $\bar\alpha=0$, and the proof will be similar to the proof of Proposition \ref{unifm-cont-welldefined}. For simplicity, define $\hat{\alpha}_{i_k}\displaystyle:=\frac{\alpha_{i_k}}{\theta}.$ Then, because $\theta \in (0,1)$, we have that $\hat{\alpha}_{i_k}>{\alpha}_{i_k}$ and \begin{equation}\label{hatato0}\lim_{k\to\infty}\hat{\alpha}_{i_k}=0.\end{equation}
Define \begin{equation}\label{hatxk}\hat{x}^{i_k}:=(I+\hat{\alpha}_{i_k}B)^{-1}(x^{i_k}-\hat{\alpha}_{i_k}Ax^{i_k})\end{equation} and for all $k\in\NN$. Hence,  
$
\hat\alpha_{i_k}\langle A_2x^{i_k}-A_2\hat{x}^{i_k},x^{i_k}-\hat{x}^{i_k}\rangle > \delta\|x^{i_k}-\hat{x}^{i_k}\|^2.
$ Using the Cauchy-Schwartz inequality, we obtain $$\delta\|x^{i_k}-\hat{x}^{i_k}\|^2 < \hat\alpha_{i_k} \|A_2x^{i_k}-A_2\hat{x}^{i_k}\|\|x^{i_k}-\hat{x}^{i_k}\|.$$ 
Dividing by $\|x^{i_k}-\hat{x}^{i_k}\|$, we get 
\begin{align}
\delta\|x^{i_k}-\hat{x}^{i_k}\| <\hat\alpha_{i_k} \|A_2x^{i_k}-A_2\hat{x}^{i_k}\|.
\label{alphaAzero}
\end{align}
Letting $k\to +\infty$, and since $(\|A_2x^{i_k}-A_2\hat{x}^{i_k}\|)_{k\in\NN}$ is bounded and \eqref{hatato0}, we have $\|x^{i_k}-\hat{x}^{i_k}\|\to 0$. Hence, $\hat{x}^{i_k}\rightharpoonup \hat x$. Since $A_2$ is uniformly continuous then \begin{equation}\label{xk-barxkto0*}\|A_2x^{i_k}-A_2\hat{x}^{i_k}\|\to 0.\end{equation} Using again \eqref{alphaAzero}, \begin{equation}\label{Akto0*}\displaystyle\frac{\|x^{i_k}-\hat{x}^{i_k}\|}{\hat\alpha_{i_k}}\to 0.\end{equation} 
The definition of $\hat x^{i_k}$ in \eqref{hatxk} together with \eqref{fermat} imply 
$$\hat{w}^{i_k}:=\displaystyle\frac{x^k-\hat{x}^{i_k}}{\hat \alpha_{i_k}}-(A_1x^{i_k}-A_1\hat{x}^{i_k})-(A_2x^{i_k}-A_2\hat{x}^{i_k})\in(A_1+A_2+B)\hat{x}^{i_k},$$ or equivalently, $$(\hat{x}^{i_k},\hat{w}^{i_k})\in \gr(A_1+A_2+B).$$ Observe that 
\begin{align*}\|\hat{w}^{i_k}\|&=\left\|\displaystyle\frac{x^{i_k}-\hat{x}^{i_k}}{\hat\alpha_{{i_k}}}-(A_1x^{i_k}-A_1\hat{x}^{i_k})-(A_2x^{i_k}-A_2\hat{x}^{i_k})\right\|\\&\le \displaystyle\frac{\|x^{i_k}-\bar{x}^{i_k}\|}{\hat\alpha_{{i_k}}}+\|A_1x^{i_k}-A_1\bar{x}^{i_k}\|+\|A_2x^{i_k}-A_2\bar{x}^{i_k}\|\\&\le \displaystyle\frac{\|x^{i_k}-\bar{x}^{i_k}\|}{\hat\alpha_{{i_k}}}+\beta\|x^{i_k}-\bar{x}^{i_k}\|+\|A_2x^{i_k}-A_2\bar{x}^{i_k}\|\end{align*}
Hence,   $\hat{w}^{i_k}$ converges to $0$ by using \eqref{Akto0*}, \eqref{limnormzero} and \eqref{xk-barxkto0*}. Since $\hat{x}^{i_k}\rightharpoonup \hat x$ and $\hat{w}^{i_k}\to 0$ and by the sequentially closedness in the weak-strong topology  of $\gr (A_1+A_2+B)$ (Lemma \ref{bound}(ii)), we have $$(\hat x,0)\in \gr(A_1+A_2+B),$$ or  equivalently,
$0\in(A_1+A_2+B)\hat x.$ Then, $\hat x\in \zer(A_1+A_2+B)$. Then, all weak accumulation points of $(x^{k})_{k\in\NN}$ belong to $\zer(A_1+A_2+B)$.
\end{proof}
\begin{theorem}\label{teo1}
The generated sequence $(x^k)_{k\in\NN}$ by {\bf Method 1} converges weakly to some element belonging to the optimal solutions set $\zer(A_1+A_2+B)$.
\end{theorem}
\begin{proof}
The result follows from the Fej\'er monotonicity of $(x^k)_{k\in\NN}$ to $\zer(A_1+A_2+B)$ proved in Proposition \ref{prop2}(i), the optimality of the weak accumulation points showed in Theorem \ref{teo1*} and Proposition \ref{punto}(iii).
\end{proof}
\subsection{Convergence of Method 2}
In this subsection all results are referent to {\bf Method 2}, i.e., with the iterative step \eqref{Fk} in {\bf Step 2.} as 
$$x^{k+1}=\mathcal{F}(x^k,\bar x^k)=\mathcal{F}_2(x^k,\bar x^k)=\Pi_{\Tau_k\cap \Gamma_k}(x^0).$$
Provided that, 
\begin{align*}x^{k+1}&=\Pi_{\Tau_k\cap\Gamma_k}(x^0)=x^0-\lambda^k_1\bar{w}^k-\lambda^k_2(x^0-x^k)\\&=(1-\lambda^k_2)(x^0-x^k)+\left(1-\frac{\lambda^k_1}{\alpha_k}\right)x^k+\frac{\lambda^k_1}{\alpha_k}\bar x^k-\lambda^k_1\left[A_2\bar x^k-A_2x^k\right].\end{align*}
where $\lambda^k_1$ and $\lambda^k_2$ are given by \eqref{l1} and \eqref{l2}. This forward step is still more general than the forward step of {\bf Method 1} and it is interesting on its own. Actually, if $\lambda^k_1=\lambda_k$ given in \eqref{beta-k} and $\lambda^k_2=1$ then it recovers the projection-forward step of {\bf Method 1}. Moreover, the projection forward step for {\bf Method 2} can be replaced by under and over projections (or even inexact projections) onto $\Tau_k\cap \Gamma_k$ producing strongly convergent versions of {\bf Method 2}. 
\begin{proposition}
If {\bf Method 2} stops then $x^k\in \zer(A_1+A_2+B)$.
\end{proposition}
\begin{proof}
If {\bf Method 2} stops then $x^{k+1}=\Pi_{\Tau_k\cap \Gamma_k}(x^0)=x^k$. So, $x^k\in \Tau_k\cap \Gamma_k\subset \Tau_k$ and it implies that, by using Proposition \ref{H<=>S^*}, $x^k\in \zer(A_1+A_2+B)$.
\end{proof}

To avoid this case, we may assume that {\bf Method 2} does not stop. Note that, $\Gamma_k$ and $\Tau_k$ are convex and closed halfspaces, for all $k$. Therefore $\Tau_k\cap \Gamma_k$ is a convex and closed set. So, if $\Tau_k\cap \Gamma_k$ is nonempty, then $x^{k+1}$ is well-defined. This notion is confirmed using the following lemma.
\begin{lemma}\label{lemma:3} $\zer(A_1+A_2+B)\subseteq \Tau_k\cap \Gamma_k$, for all $k\in\NN$.
\end{lemma}
\begin{proof} Note that, the set of optimal solutions $\zer(A_1+A_2+B)\neq\emptyset$. By Proposition \ref{S*subsetH}, $\zer(A_1+A_2+B)\subseteq \Tau_k$, for all $k\in\NN$.
When $k=0$, it is the case $\Tau_0=\HH$, we get $\zer(A_1+A_2+B)\subset \Tau_0\cap \Gamma_0$.
For all $\ell\in\NN$ such that $\ell\leq k$, we have by induction hypothesis that $\zer(A_1+A_2+B)\subset \Tau_\ell\cap \Gamma_\ell$. Therefore, $x^{k+1}=\Pi_{\Tau_k\cap \Gamma_k}(x^0)$ is well-defined.
Then, the following inequality is consequence of the induction hypothesis and by Proposition \ref{proj}(ii), i.e.,
\begin{equation}\label{x*ink+1}
\langle x^*-x^{k+1},x^0-x^{k+1}\rangle=\langle x^*-\Pi_{\Tau_k\cap \Gamma_k}(x^0),x^0-\Pi_{\Tau_k\cap \Gamma_k}(x^0)\rangle\leq0,
\end{equation}
for all $x^*\in \zer(A_1+A_2+B)$. Notice that \eqref{x*ink+1} and \eqref{W(x)} imply that
$x^*\in \Gamma_{k+1}$ and hence, $ \zer(A_1+A_2+B)\subset \Tau_{k+1}\cap \Gamma_{k+1}$. So, the result follows.
\end{proof}

We have shown that the set $\Tau_k\cap \Gamma_k$ is  nonempty and therefore the projection step, given in  \eqref{iterado3}, is well-defined.
\begin{corollary}\label{l:well-definedness}  {\bf Method 2} is well-defined.
\end{corollary}
\begin{proof} By Lemma \ref{lemma:3} , $ \zer(A_1+A_2+B)\subset \Tau_k\cap \Gamma_k$, for  all $k\in\NN$. Then,  for a given initial $x^0$, the sequence $(x^k)_{k\in\NN}$ is attainable.
\end{proof}

To prove the convergence of the sequence, we need first to show conditions and bounds on the sequence. Next lemma tackles this issue by restricting the sequence in a ball defined by its initial point $x^0$.
\begin{lemma}\label{seq-bdd} The sequence $(x^k)_{k\in\NN}$ is bounded. Moreover, 
\begin{equation*}\label{eq:bolas}
(x^k)_{k\in\NN}\subset  \mathbb{B}\left[\frac{1}{2}(x^0+\bar{x});\frac{1}{2}\rho\right],
\end{equation*} where $\bar{x}:=\Pi_{\zer(A_1+A_2+B)}(x^0)$ and $\rho:={\rm dist}(x^0,\zer(A_1+A_2+B))$. 
\end{lemma}
\begin{proof} 
Using Lemma \ref{lemma:3} we see that $\zer(A_1+A_2+B)\subseteq \Tau_k\cap \Gamma_k$, and by \eqref{iterado3}, we obtain that
\begin{equation*}
\| x^{k+1}-x^0\|\leq\|x^*-x^0\|,
\end{equation*}
for all $k\in\NN$ and all $x^*\in \zer(A_1+A_2+B)$. Taking $x^*=\bar{x}$ in the above inequality, we have
\begin{equation*}
\|x^{k+1}-x^0\|\leq\|\bar{x}-x^0\|=\rho,\quad\text{for all}\quad k\in\NN.
\end{equation*}
Proving that $(x^k)_{k\in\NN}$ is bounded. \\
Now, define $x^*_{k}:=x^{k}-\frac{1}{2}(x^0+\bar{x})$ and $\bar{x}^*:=\bar{x}-\frac{1}{2}(x^0+\bar{x})$. Since $\bar{x}\in \Gamma_{k+1}$, then we have 
\begin{align*}
0&\geq2\big\langle\bar{x}-x^{k+1},x^0-x^{k+1}\big\rangle \\&=2\left\langle\bar{x}^*+\frac{1}{2}(x^0+\bar{x})-x^*_{k+1} -\frac{1}{2}(x^0+\bar{x}),x^*_{0}+\frac{1}{2}(x^0+\bar{x})-x^*_{k+1}-\frac{1}{2}(x^0+\bar{x})\right\rangle\\&=2\left\langle\bar{x}^*-x^*_{k+1},x^*_{0}-x^*_{k}\right\rangle=\left\langle\bar{x}^*-x^*_{k+1},-\bar{x}^*-x^*_{k+1}\right\rangle=\|x^*_{k+1}\|^2-\|\bar{x}^*\|^2,
\end{align*}
Thus, after manipulating, we obtain 
\begin{equation*}\label{eq:raio}
\left\|x^{k+1}-\frac{x^0+\bar{x}}{2}\right\|\leq\left\|\bar{x}-\frac{x^0+\bar{x}}{2}\right\|=\frac{\rho}{2},\quad\text{for all}\quad k\in\NN.
\end{equation*}
Hence, the result follows and $(x^k)_{k\in\NN}$ lies in the ball determined by the initial point $x^0$.
\end{proof}

An important question arisen here is, whether the set of optimal solutions contains any weak accumulation point of $(x^k)_{k\in\NN}$ or not. Next lemma answers this question by showing that all weak accumulation points must lie in the optimal solutions set, $\zer(A_1+A_2+B)$.
\begin{lemma}\label{all-weak-accum} Let $(x^k)_{k\in\NN}$ be a sequence generated by {\bf Method 2}. Then, all weak accumulation points of the sequence $(x^k)_{k\in\NN}$ belong to $\zer(A_1+A_2+B)$.
\end{lemma}
\begin{proof}
If $x^{k+1}\in \Gamma_k$, then 
\begin{align*}
0&\ge 2\la x^{k+1}-x^k,x^0-x^k\ra = \|x^{k+1}-x^k\|^2-\|x^{k+1}-x^0\|^2+\|x^k-x^0\|^2.
\end{align*}
which implies $0\leq\|x^{k+1}-x^k\|^2\leq\|x^{k+1}-x^0\|^2 -\|x^k-x^0\|^2.$ Hence, $(\|x^k-x^0\|)_{k\in\NN}$ is a nondecreasing sequence. Since the sequence $(\|x^k-x^0\|)_{k\in\NN}$ is bounded, by Lemma \ref{seq-bdd}, it is therefore convergent. Thus,
\begin{equation}\label{xk+1-xk-to-zero}
\lim_{k\rightarrow\infty}\| x^{k+1}-x^k\|=0.
\end{equation}
Since the projection of $x^k$ onto the halfspace $\Tau_k$ is denoted by $\Pi_{\Tau_k}(x^k)$, then we have the inequality $$0\leq\|x^k-\Pi_{\Tau_k}(x^k)\|\leq\|x^k-x\|\qquad\text{for all}\quad x\in \Tau_k.$$
The fact that $x^{k+1}\in \Tau_k$ implies $0\leq\|x^k-\Pi_{\Tau_k}(x^k)\|\leq\|x^k-x^{k+1}\|$. Since  $\|x^k-x^{k+1}\|\to 0$ by \eqref{xk+1-xk-to-zero}, we have
\begin{align*}
0&=\lim_{k\rightarrow\infty}\|x^k-\Pi_{\Tau_k}(x^k)\|=\lim_{k\to\infty}\frac{\Big\langle\displaystyle\frac{x^k-\bar{x}^k}{\alpha_k}-(A_2x^k-A_2\bar{x}^k),x^k-\bar{x}^k\Big\rangle-\frac{\bar\delta}{\alpha_k}\|x^k-\bar x^k\|^2}{\left\|\displaystyle\frac{x^k-\bar{x}^k}{\alpha_k}-(A_2x^k-A_2\bar{x}^k)\right\|}\\&=\lim_{k\to\infty}\frac{\Big\langle x^k-\bar{x}^k-{\alpha_k}(A_2x^k-A_2\bar{x}^k),x^k-\bar{x}^k\Big\rangle-\bar\delta\|x^k-\bar x^k\|^2}{\|\displaystyle x^k-\bar{x}^k-{\alpha_k}(A_2x^k-A_2\bar{x}^k)\|}.
\end{align*}
The sequence $\left(\|\displaystyle x^k-\bar{x}^k-{\alpha_k}(A_2x^k-A_2\bar{x}^k)\|\right)_{k\in\NN}$ is bounded because the sequence $(x^k)_{k\in\NN}$ and $(\bar x^k)_{k\in\NN}$ are bounded and 
$
\|\displaystyle x^k-\bar{x}^k-{\alpha_k}(A_2x^k-A_2\bar{x}^k)\|\le \|\displaystyle x^k-\bar{x}^k\|+{\alpha_k}\|A_2x^k-A_2\bar{x}^k\|.
$
Hence, 
\begin{equation*}
0=\lim_{k\to\infty}\langle x^k-\bar{x}^k-{\alpha_k}(A_2x^k-A_2\bar{x}^k),x^k-\bar{x}^k\rangle-\bar\delta\|x^k-\bar x^k\|^2\ge (1-\delta-\bar\delta) \lim_{k\to\infty} \|x^k-\bar x^k\|^2\ge0,
\end{equation*} using \eqref{linesearch} in the first inequality. Hence, 
$
\lim_{k\to \infty}\|x^{k} - \bar x^{k}\|=0.
$
Now the rest of the proof follows similarly to the proof of Theorem \ref{teo1*} after \eqref{limnormzero}.
\end{proof}

Finally, we prove the main theorem of convergence of the sequence $(x^k)_{k\in\NN}$ generated by {\bf Method 2}, which converges to the nearest solution to $x^0$.
\begin{theorem} Let $(x^k)_{k\in\NN}$ be a sequence generated by {\bf Method 2}. Define $\bar{x}:=\Pi_{\zer(A_1+A_2+B)}(x^0)$. Then, $(x^k)_{k\in\NN}$ converges strongly to $\bar{x}$.
\end{theorem}
\begin{proof} 
Since $\zer(A_1+A_2+B)$ is nonempty closed convex by Lemma \ref{bound}(i) then there exists a metric (orthogonal) projection of $x^0$ onto $\zer(A_1+A_2+B)$, i.e. $\bar{x}=\Pi_{\zer(A_1+A_2+B)}(x^0).$ Since $\bar{x}\in \zer(A_1+A_2+B)\subseteq \Tau_k\cap \Gamma_k$, and by using the projection definition of $x^{k+1}$ onto the intersection of the hyperplanes $\Tau_k\cap \Gamma_k$ as in \eqref{iterado3}, we then have the property 
\begin{equation}\label{bound-barx}\|x^k-x^0\|\leq\|\bar{x}-x^0\|\quad\text{for all}\quad k\in\NN.\end{equation} Since $(x^k)_{k\in\NN}$ is bounded by Lemma \ref{seq-bdd}, and every weak accumulation point of the sequence $(x^k)_{k\in\NN}$ is in the set of optimal solutions $\zer(A_1+A_2+B)$ by Lemma \ref{all-weak-accum}, then any subsequence $(x^{i_k})_{k\in\NN}$ of $(x^k)_{k\in\NN}$ must converge weakly to an accumulation point belonging to $\zer(A_1+A_2+B)$ (say $x^{i_k}\rightharpoonup\hat{x}\in \zer(A_1+A_2+B)$). Thus, 
\begin{align*}
\|x^{i_k}-\bar{x}\|^2&=\|x^{i_k}-x^0-(\bar{x}-x^0)\|^2 \\ &=\|x^{i_k}-x^0\|^2+\|\bar{x}-x^0\|^2-2\la x^{i_k}-x^0,\bar{x}-x^0\ra \\&\leq 2\|\bar{x}-x^0\|^2-2\la x^{i_k}-x^0,\bar{x}-x^0\ra,
\end{align*} using \eqref{bound-barx} in the last inequality.
Since the subsequence $(x^{i_k})_{k\in\NN}$ is weakly convergent to $\hat{x}$, we get
\begin{align*}\label{limsupwklyconv}
\lim\sup_{k\to\infty}\|x^{i_k}-\bar{x}\|^2&\leq 2\|\bar{x}-x^0\|^2-2\la\hat{x}-x^0,\bar{x}-x^0\ra\\&=2\la\bar{x}-\hat{x},\bar{x}-x^0\ra+2\la\hat{x}-x^0,\bar{x}-x^0\ra-2\la\hat{x}-x^0,\bar{x}-x^0\ra\\&=2\la\bar{x}-\hat{x},\bar{x}-x^0\ra\\&\leq 0,
\end{align*} where the last inequality follows from the definition of $\bar x$ and Proposition \ref{proj}(ii). This implies that $x^{i_k}\to\bar{x}.$ Hence, we have proved that every weakly convergent subsequence of the generated sequence $(x^k)_{k\in\NN}$ converges strongly to $\bar{x}\in\zer(A_1+A_2+B).$ Hence, the whole sequence $(x^k)_{k\in\NN}$ converges strongly to $\bar{x}=\Pi_{\zer(A_1+A_2+B)}(x^0)$ proving the desired result.
\end{proof}
\section{Concluding Remarks}\label{section5}
This paper dealt with the weak and strong convergence of a modification of the forward-backward-forward (FBF) splitting method for solving monotone inclusions. We propose a conceptual algorithm modifying the forward-backward-half-forward (FBHF) splitting algorithm for finding a zero to the sum of three monotone operators. The two proposed modified variants use a backtracking strategy to find a suitable separating hyperplane, which splits the space into two halfspaces, one containing the optimal solution set and one containing the current iterate. Two explicit projection steps onto suitable halfspaces produce two methods with desired features. It is worth emphasizing that one of the variants is strongly
convergent to the best approximation solution, and both proposed methods produce very general iterations recovering the FBHF splitting method. 

We finalize by mentioning that replacing the Euclidean norm in the projection (forward) steps by a more general norm or Bregman distance may produce variants of the proposed iterations with a certain interest in some applications. Therefore, extending the presented analysis for more general projection steps
could be a promising subject for future research.

\medskip

\noindent{\bf Acknowledgments:} YBC was partially supported by the National Science Foundation (NSF), Grant DMS - 1816449. 

\medskip 

\noindent {\bf Data Availability Statement:} Data sharing not applicable to this article as no datasets were generated or analyzed during the current study.

\bibliographystyle{plain}

\end{document}